\documentclass[10pt,a4paper,reqno]{amsart}
\usepackage{amsthm,amssymb,amsmath,graphicx,enumerate}
\usepackage{mathrsfs,setspace,pstricks,multicol,latexsym}
\usepackage{epsfig, subcaption, graphics,xcolor}
\usepackage{comment}
\usepackage{hyperref}
\usepackage{multirow}
\usepackage{dsfont}
\usepackage{float}
\usepackage{caption}

  \usepackage{times}
\numberwithin{equation}{section}
\newtheorem{theo}{{\bf{Theorem}}}[section]
\newtheorem{prop}[theo]{{\bf Proposition}}
\newtheorem{lem}[theo]{{\bf Lemma}}

\newtheorem{rem}[theo]{{\bf Remark}}

\newtheorem{defi}[theo]{{\bf Definition}}

\newtheorem{notation}[theo]{{\bf Notation}}

\renewcommand{\proof}{\noindent{\bf Proof.\ }}

\newcommand{\no}{\nonumber}
\newcommand{\noi}{\noindent}

\hypersetup{
    colorlinks=true,
    linkcolor=blue,
    filecolor=magenta,
    urlcolor=cyan,
    pdftitle={Overleaf Example},
    pdfpagemode=FullScreen,
    }
\catcode`\@=11
\catcode`\@=12

\begin{document}
\title{Semimartingle Representation of a class of Semi-Markov Dynamics}\thanks{This research was supported in part by NBHM 02011/1/2019/NBHM(RP)R\& D-II/585, DST/INTDAAD/P-12/2020, and DST FIST (SR/FST/MSI-105). The second author would like to acknowledge the support from CSIR SRF. The support and the resources provided by ‘PARAM Brahma Facility’ under the National Supercomputing Mission, Government of India at the Indian Institute of Science Education and Research (IISER) Pune are gratefully acknowledged.}

\author{Anindya Goswami*}
\address{IISER Pune, India}
\email{anindya@iiserpune.ac.in}
\thanks{*Corresponding author}
\author{Subhamay Saha}
\address{IIT Guwahati, India}
\email{saha.subhamay@iitg.ac.in}
\author{Ravishankar Kapildev Yadav}
\address{IISER Pune, India}
\email{ravishankar.kapildevyadav@students.iiserpune.ac.in}

\maketitle

%\addtocounter{footnote}{-1} \vskip 1 true cm
\begin{abstract}
We consider a class of semi-Markov processes (SMP) such that the embedded discrete time Markov chain may be non-homogeneous. The corresponding augmented processes are represented as semi-martingales using stochastic integral equation involving a Poisson random measure. The existence and uniqueness of the equation are established. Subsequently, we show that the solution is indeed a SMP with desired transition rate. Finally, we derive the law of the bivariate process obtained from two solutions of the equation having two different initial conditions.
\end{abstract}

{\bf Key words} Poisson random measure, non-homogeneous Semi-Markov Processes, semi-Markov system\\
{\bf AMSC} 60G55, 60H20, 60K15
\section{Introduction}

\noindent Apparently \cite{levyp} and \cite{smith1955regenerative}, which were presented at the International Congress of Mathematicians held at Amsterdam in 1954, are the first available literature that discuss the mathematical aspects of semi-Markov processes (SMP). In L\'{e}vy's work, \cite{levyp},  the definition of SMP was presented as a generalization of Markov chain. Around the same time, independent to L\'{e}vy's work, Smith \cite{smith1955regenerative} and then Tackas \cite{takacs1954some} have also introduced SMP. Below we present a modern version of the definition.
\begin{defi} \label{def1}
$X=\{X_{t}\}_{t\geq 0}$ defined on a complete probability space $(\Omega,\mathcal{B},P)$ is an SMP with state space $\mathcal{X}$ if
\begin{enumerate}
\item $X$ is piece-wise constant r.c.l.l. process with discontinuities at $T_{1}<T_2 <\cdots< T_{n}<\cdots$,
\item and for each $n\ge 1,j\in\mathcal{X},$ and $ y>0,$ {\small\begin{equation}\label{defi_semimarkov}
    P[X_{T_{n+1}}=j,T_{n+1}-T_{n}\leq y\mid (X_0, T_0), (X_{T_k},T_{k})\; \forall 1\le k\le n]=P[X_{T_{n+1}}=j,T_{n+1}-T_{n}\leq y \mid X_{T_{n}}]
\end{equation}}
where $T_{0}\le 0< T_1$.
\end{enumerate}
$X$ is pure if $T_n \xrightarrow{a.s.}\infty$ as  $n\rightarrow \infty$. If, the right side of \eqref{defi_semimarkov} is independent on $n$, then the SMP $X$ is called time homogeneous. Otherwise, $X$ is called non-homogeneous SMP.  Here, $T_n$ is called $n$\textsuperscript{th} transition time.
\end{defi}
\noindent We recall that in Chapter 1 of \cite{boris.H2008}, a book by Boris Harmalov, a stepped SMP is introduced and in subsequent chapters further generalizations to continuous state space appears. We confine ourselves to the study of SMPs on a countable state space. The continuous-time discrete-state homogeneous Markov chains are included in this class. The class under consideration is such that the embedded discrete time Markov chain may also be non-homogeneous. However, a more general non-homogeneity that appears in a continuous time Markov process having time dependent transition rate is not included.

\noindent Unlike Definition \ref{def1}, SMP has been defined by many authors using renewal processes. For example, in \cite{pyke1961markov}, Pyke has introduced the SMPs by specifying the conditional distribution of next state and holding time given the past states and past holding times. We notice that the $\sigma$-algebra generated by the transition times is identical to that generated by the initial time and past holding times. Thus the conditional distribution in \cite{pyke1961markov} is identical to the conditional probability in Definition \ref{def1}. In \cite{pyke1961markov}, a concept of regularity has been introduced for assuring that the chain is pure. Besides, classification of the states has also been studied there. On the other hand in  \cite{pyke21961markov}, by considering the finite state SMPs, Pyke has derived expressions for the distribution functions of first passage times, as well as for the marginal distribution function.
Furthermore, the limiting behavior of a Markov Renewal process has been discussed, and the stationary probabilities have been derived. Various aspects and approaches regarding limiting behavior has also been studied by Orey \cite{Orey1961CHANGEOT}, around the same time. In \cite{bennetfox}, Bennet has studied some properties of sub-chain of SMP, which are obtained via regenerating points, if they exist.

\noindent It has been noted by several authors (see Nummelin \cite{nummelin_1976}, Athreya et al \cite{Athreya1978} and references therein) that an SMP can be augmented with the age process to obtain a jointly Markov process whose Feller property and infinitesimal generator can be derived (see Chapter 2, \cite{gihman}). In a recent work \cite{elliott2020semimartingale} Elliott has presented a semimartingale dynamics of semi-Markov chain in contrast to the traditional description of a semi-Markov chain in terms of a renewal process. This presentation is different from that in \cite{Ghosh2009RiskMO} and \cite{GhoshSaha2011} where a semimartingale representation appears using an integration with respect to a Poisson random measure(PRM). Such semimartingale representations are useful for studying several aspects including the stochastic flow of semi-Markov dynamics.

Apparently, SMP beyond the class of time homogeneity has been first studied in \cite{Hoem1972InhomogeneousSP} and \cite{manu1972}
Various different aspects and generalisations of non-homogeneous SMP(NHSMP) has further been explored in \cite{JANSSEN1984,transientcase,janssen2006applied,janssen2007semi}. In these references several applications of NHSMP has also been emphasised.

\noindent In this paper we consider a system of SDEs driven by a PRM and having coefficients depending on a given transition rate function and an additional gaping parameter. It is worth noting that neither the coefficients are compactly supported nor the intensity measure is finite. So, we produce a self-contained proof of the existence and uniqueness of the solution to the SDE. We further show that the state component of the solution is a pure semi-Markov process with the given transition rate function. The law of a single solution of course do not depend on the gaping parameter. However, we show, the joint distribution of a couple of solutions with different initial conditions do depend on the additional gaping parameter. The combination of state processes of two solutions forms a semi-Markov system(SMS) or a component-wise semi-Markov (CSM) process having dependent components. The SMSs \cite{vasileiou_vassiliou_2006,exotic,vassiliou1992},  or CSMs \cite{milan_pricing_derivative, MAN} with independent components have been introduced for modelling some random dynamics. However, a CSM with dependent components has not been studied in the literature yet. The law of the CSM has been calculated in terms of the infinitesimal generator of the augmented process. As per our knowledge, the present paper is the first work on the correlated semi-Markov system. The immense applicability of semi-Markov processes and semi-Markov systems is well known. So, the formulation and the results related to the correlated semi-Markov system, presented in this paper have significantly expanded the scope of further theoretical and applied studies. In particular this has opened up the possibility of studying a system generated by a semi-Markov flow. The questions related to the meeting and merging events of multiple particles of a semi-Markov flow have far reaching implications. In view of these, the present paper appears as a stepping stone for several future studies.

\noindent The rest of this paper is arranged in the following manner. We present a class of SMPs as semi-martingales using an SDE involving PRM in Section 2. Here we first consider the SDE and then we show that it has almost sure unique solution. The state component is shown to possess semi-Markov law in Section 3. The  transition parameters of the semi-Markov process are also obtained in that section. In Section 4 the joint distribution of solutions having two different initial conditions has been derived in terms of the infinitesimal generator of the combined process. Some concluding remarks and the scope of future research directions are added in Section 5.

\section{SDE using Poisson Random Measure}
\noindent Let $(\Omega,\mathcal{F},\{\mathcal{F}_t\}_{t\ge0},P)$ be the underlying filtered probability space satisfying the usual hypothesis and $\mathcal{X}=\{1,2,\mathellipsis\}$ the state space. We wish to construct a semi-Markov process on this state space with a given transition rate function. To this end we first embed $\mathcal{X}$ in $\mathbb{R}$ and endow with a total order $\prec_1$, which in turn induces a total order $\prec_2$ on $\mathcal{X}_{2}:=\{(i,j)\in \mathcal{X}^{2}\ \mid  \ i\neq j\}$ by following lexicographic order. Let $\mathcal{B}(\mathbb{R}^{d})$ denote the Borel $\sigma$-algebra on $\mathbb{R}^{d}$ and $m_{d}$ denote the Lebesgue measure on $\mathbb{R}^{d}.$ Also let  $\mathbb{N}_{0}$ and $\mathbb{R}_{+}$ denote the set of non-negative integers and the set of non-negative real numbers respectively. For each $y\in \mathbb{R}_{+}$,  $n\in \mathbb{N}_{0}$, $\lambda(y,n):=(\lambda_{ij}(y,n))$ denote a matrix in which the $i\textsuperscript{th}$ diagonal element is negative of $\lambda_{i}(y,n) :=\sum_{j\in\mathcal{X}\setminus\{i\}}\lambda_{ij}(y,n)$
and for each $(i,j)\in \mathcal{X}_2$, $ \lambda_{ij}\colon\mathbb{R}_{+}\times \mathbb{N}_0\to\mathbb{R}_{+}$ is a bounded measurable function satisfying the following two assumptions.
\begin{enumerate}
\item[{\bf (A1)}.]  If $c_i:=\underset{j\in\mathcal{X}\setminus\{i\}}{\sum}\|\lambda_{ij}(\cdot,\cdot)\|_{L^{\infty}_{(\mathbb{R}_+\times\mathbb{N}_0)}}$, $c:=\sup_{i}c_i<\infty$.
\item[{\bf (A2)}.] For each $n$ in $\mathbb{N}_0$, and $i$ in $\mathcal{X}$, $\lim_{y\rightarrow\infty}\gamma_{i}(y,n) =\infty$, where $ \gamma_{i}(y,n):=\int_{0}^{y}\lambda_{i}(y',n)dy'$.
\end{enumerate}

\noindent  For each $(i,j)\in \mathcal{X}_{2},$ we consider another measurable function $\tilde{\lambda}_{ij}\colon\mathbb{R}_{+}\times\mathbb{N}_0\to \mathbb{R}_{+}$ such that for each $y\in \mathbb{R}_{+}$, and $n\in \mathbb{N}_0$
\begin{align} \label{a}
\lambda_{ij}(y,n)\leq\tilde{\lambda}_{ij}(y,n),
\end{align}
and also for almost every $y \in \mathbb{R}_{+}$ and $n\in \mathbb{N}_0$
\begin{align}\label{tildeinequality1}
\tilde{\lambda}_{ij}(y,n)\leq \|\lambda_{ij}(\cdot,\cdot)\|_{L^{\infty}_{(\mathbb{R}_+\times\mathbb{N}_0)}}.
\end{align}
Now for  each $y\in \mathbb{R}_{+}$, and $n\in \mathbb{N}_0$, with the help of $\lambda(y,n)$ and $\tilde{\lambda}(y,n):=(\tilde\lambda_{ij}(y,n))$, we introduce a disjoint collection of intervals $\Lambda:= \{\Lambda_{ij}(y,n)\colon (i,j)\in \mathcal{X}_{2}\}$, by
\begin{align}\label{genricinterval}
\Lambda_{ij}(y,n)=\left(\sum_{(i',j')\prec_2(i,j)}\tilde{\lambda}_{i'j'}(y,n)\right)\ +\Big[0,\lambda_{ij}(y,n)\Big)
\end{align}
where $a+B=\{a+b\mid b\in B\}$ for $a\in \mathbb{R}, B\subset \mathbb{R}$. Clearly, for each $y$ and $n$, the interval $\Lambda_{ij}(y,n)$ is of length $\lambda_{ij}(y,n)$. If $C_i:= c_i + \sum_{k \prec_1 i}c_k$, according to (A1) and \eqref{tildeinequality1}, the set $\Lambda_i(y,n):= \cup_{j\in \mathcal{X}\setminus\{i\}} \Lambda_{ij}(y,n)$  is contained in the finite interval $\left[0,C_i\right]$, for almost every $y\in \mathbb{R}_{+}$ and each $n$.
Using the above intervals we define  $h_{\Lambda} \colon \mathcal{X}\times \mathbb{R}_{+} \times \mathbb{N}_0\times \mathbb{R}\rightarrow \mathbb{R}$ as
	 \begin{equation}\label{h}
	     h_{\Lambda}(i, y, n, v):=\sum_{j\in \mathcal{X}\setminus\{i\}}\!(j-i)\mathds{1}_{\Lambda_{ij}(y,n)}(v)
	 \end{equation}
	   and $g_{\Lambda} \colon \mathcal{X}\times \mathbb{R}_{+} \times \mathbb{N}_0\times\mathbb{R}\rightarrow \mathbb{R}$ as
	   \begin{equation}\label{g}
	   g_{\Lambda}(i, y, n, v):=\sum_{j\in \mathcal{X}\setminus\{i\}}\!\mathds{1}_{\Lambda_{ij}(y,n)}(v).
	   \end{equation}
Here $\mathds{1}_B$ is the indicator function of the set $B$. These functions are piece-wise constant in $v$ variable. Using these, we consider the following system of coupled stochastic integral equations in $X$, $Y$ and $N$
     \begin{align}
         \label{eqX}
	     X_{t}&=X_{0}+\displaystyle \int_{0+}^{t}\int_{\mathbb{R}}h_{\Lambda}(X_{u-}, Y_{u-}, N_{u-},v)\wp(du,dv)\\
	   	\label{eqY}
	   	Y_{t}&=Y_{0}+t-\displaystyle \int_{0+}^{t}(Y_{u-})\int_{\mathbb{R}}g_{\Lambda}(X_{u-}, Y_{u-},N_{u-},v)\wp(du, dv)\\
	   	\label{eq_N(t)}
	   	N_{t}&=\int_{0+}^{t}\int_{\mathbb{R}}g_{\Lambda}(X_{u-},Y_{u-},N_{u-},v)\wp(du,dv)
   	\end{align}
for $t\ge 0$, where $\int_{0+}^{t}$ is integration over the interval $(0,t]$, $\wp(du,\ dv)$  is a PRM on $\mathbb{R}_{+}\times \mathbb{R}$ with intensity $m_{2}(du,dv)$, and defined on the probability space $(\Omega,\mathcal{F},P).$ We also assume that $\{\wp((0,t] \times dv)\}_{t\geq 0}$ is adapted to $\{\mathcal{F}_{t}\}_{t\geq 0}$.
For $t\ge 0$, we rewrite the above coupled integral equation \eqref{eqX}-\eqref{eq_N(t)} as a vector form given by,
\begin{align}\label{vectorform}
    Z_t=Z_0+\int_0^ta(Z_{u-})du+ \int_{0+}^t\int_{\mathbb{R}} J(Z_{u-},v)\wp (du,dv),
\end{align}
where for each $t\ge 0$, $Z_t=(X_t,Y_t,N_t)$, $N_0=0$, $z=(i,y,n)$, $a(z)=(0,1,0)$, and $J(z,v)$ is the vector $(h_{\Lambda}(z,v), -yg_{\Lambda}(z,v), g_{\Lambda}(z,v))$.
\begin{rem}
The existing results on the general theory of SDE are not directly applicable for assuring existence and uniqueness of solution to the system of integral equations \eqref{eqX}-\eqref{eq_N(t)}. We recall that Theorem 3.4 (p-474) of \cite{cinlar2011probability} assumes compact support of the integrands whereas Theorem IV.9.1 (p-231) of \cite{Ikeda_1981} assumes finiteness of the intensity measure on the complement of a neighbourhood of origin. For the system \eqref{eqX}-\eqref{eq_N(t)} neither the integrands are compactly supported nor the intensity of Poisson random measure is finite in the complement of any bounded set. For this reason we produce an original proof of existence and uniqueness of \eqref{eqX}-\eqref{eq_N(t)}.
\end{rem}

\begin{theo}\label{theo3.1}
Under assumption (A1), there exists an increasing sequence of stopping times $\{T_n\}_{n=1}^{\infty}$ such that the following hold.
\begin{enumerate}
\item The coupled system of stochastic integral equations \eqref{eqX}-\eqref{eq_N(t)} has a unique strong solution $(X,Y,N)={(X_t,Y_t,N_t)}_{t\in [0,\tau)}$ where $\tau=\lim_{n\to\infty}T_n$.
\item Almost surely $X$, $Y$, and $N$ have r.c.l.l. paths respectively and they jump only at $T_n$, for each $n$. While $X$ and $N$ are piece-wise constants, $Y$ is piece-wise linear.
\item We set $T_{0} := -Y_{0}$.  For each $n\in \mathbb{N}$, (i) $Y_{T_n}=0$, (ii) $Y_{T_{n}-} = T_{n}- T_{n-1}$, and (iii) $N_{T_n}=n$. Also (iv)  $Y_t =t-T_{N_t}$ for all $t\in \mathbb{R}_{+}$.
\end{enumerate}
\end{theo}
\begin{proof}
We recall that due to Assumption (A1) and \eqref{tildeinequality1}, for almost every $y\in \mathbb{R}_{+}$, and each $i,n$ the union of  intervals $\{ \Lambda_{ij}(y,n) \mid {j\in \mathcal{X}\setminus\{i\}} \}$ is contained in $[0, C_i]$. For each $\omega\in \Omega$ and $i\in \mathcal{X}$ we define the set $\mathcal{D}_i := \{s\in (0, \infty)\mid \wp(\omega)(\{s\}\times[0, C_i]) > 0\} $, the collection of time coordinates of the point masses of a realisation of the PRM $\wp(\omega)$ in which the second component is not more than $C_i$. Since, the interval $[0, C_i]$ has finite Lebesgue measure, by Lemma A.1, $\mathcal{D}_i$ has no limit points in $\mathbb{R}$ with probability 1. Thus for each $i\in \mathcal{X}$ we can enumerate $\mathcal{D}_i$, in increasing order
\begin{equation}\label{enumerate_i}
    \mathcal{D}_i = \{\sigma^i_l\}_{l=1}^{\infty},\textrm{ where } \sigma^i_1<\cdots< \sigma^i_l < \sigma^i_{l+1}<\cdots \textrm{ for each }\omega.
\end{equation} For each $l\in \mathbb{N}, \sigma^i_l\colon\Omega\to (0,\infty]$ and  $\{\sigma^i_l\leq t\} = \{\omega \mid \wp((0, t]\times[0, C_i])\ge l\}\in \mathcal{F}_t$, as $\wp((0, t]\times[0, C])$ is $\mathcal{F}_t$ measurable. Hence $\sigma^i_l$ is a stopping time for each $i\in \mathcal{X}$  and $l\ge 1$.
As $\mathcal{D}_i$ has no limit points, $\sigma^i_l \uparrow \infty$ as $l\to \infty$ almost surely.

\noindent Here we plan to construct an increasing sequence $\{\Bar{\sigma}_m\}_{m=1}^{\infty}$ so that we can first define a solution to equations \eqref{eqX}-\eqref{eq_N(t)} on the time interval $[0,\ \Bar{\sigma}_{1}]$ and then on the next time interval $(\Bar{\sigma}_{1},\ \Bar{\sigma}_{2}]$, and so on. For a fixed $\omega\in \Omega$, $X$, $Y$ at time $t=0$ are $X_0 \in \mathcal{X}$, and $Y_0\in \mathbb{R}_+$ respectively. Using \eqref{enumerate_i}, we consider the increasing sequence $\mathcal{D}_{X_0}$ and call the first element of this sequence as $\Bar{\sigma}_1$. Hence,
$$\wp(\omega)\left([0,\ \Bar{\sigma}_{1})\times\ [0,\ C_{X_0}]\right)=0,$$
and thus  for $ t\in [0,\ \Bar{\sigma}_{1})$
\begin{align*}
    X_{t}(\omega)&=X_{0}+\int_{0+}^{t}\int_{0}^{C_{X_0}}h_{\Lambda}(X_{u-},\ Y_{u-},N_{u-},\ v)\wp(\omega)(du,\ dv) =X_{0},\\
    Y_{t}(\omega)&=Y_{0}+t-\int_{0+}^{t}\int_{0}^{C_{X_0}}(Y_{u-})g_{\Lambda}(X_{u-},\ Y_{u-},N_{u-},\ v)\wp(\omega)(du,\ dv)=Y_{0}+t,\\
    N_{t}(\omega)&=\int_{0+}^{t}\int_{0}^{C_{X_0}}g_{\Lambda}(X_{u-},\ Y_{u-},N_{u-},\ v)\wp(\omega)(du,\ dv)=0.
\end{align*}
Note that in the above integrations the domain $\mathbb{R}$ has been replaced by the compact set $[0,C_{X_0}]$, as the integrands vanish outside this interval. Hence at $t=\Bar{\sigma}_{1},$
\begin{align*}
    X_{\Bar{\sigma}_{1}}(\omega)&=X_{0}+\int_{0}^{C_{X_0}}h_{\Lambda}(X_{0},\ Y_{\Bar{\sigma}_{1}-},N_{0},\ v)\wp(\omega)(\{\Bar{\sigma}_{1}\}\times dv),\\
    Y_{\Bar{\sigma}_{1}}(\omega)&=Y_{0}+\Bar{\sigma}_{1}-\int_{0}^{C_{X_0}}(Y_{\Bar{\sigma}_{1}-})g_{\Lambda}(X_{0},\ Y_{\Bar{\sigma}_{1}-},N_{0},\ v)\wp(\omega)(\{\Bar{\sigma}_{1}\}\times dv),\\
    N_{\Bar{\sigma}_{1}}(\omega)&=\int_{0}^{C_{X_0}}g_{\Lambda}(X_{0},\ Y_{\Bar{\sigma}_{1}-},N_{0},\ v)\wp(\omega)(\{\Bar{\sigma}_{1}\}\times \ dv).
\end{align*}
Hence the solution is unique in the time interval $[0,\ \Bar{\sigma}_{1}]$. Again by using \eqref{enumerate_i}, we have an increasing sequence $\mathcal{D}_{X_{\Bar{\sigma_1}}}$. As  $\mathcal{D}_{X_{\Bar{\sigma}_1}}\cap(\Bar{\sigma}_1,\infty)$ is nonempty almost surely, there is the first element of the set, denoted by $\Bar{\sigma}_2$.  Thus we have,
$$\wp(\omega)\left((\Bar{\sigma}_1,\ \Bar{\sigma}_{2})\times\ [0,\ C_{X_{\Bar{\sigma}_{1}}}]\right)=0$$
and for $ t\in (\Bar{\sigma}_1,\ \Bar{\sigma}_{2})$, as before
\begin{align*}
    X_{t}(\omega)&=X_{\Bar{\sigma}_{1}}+\int_{\Bar{\sigma}_{1}+}^{t}\int_{0}^{C_{X_{\Bar{\sigma}_{1}}}}h_{\Lambda}(X_{u-},\ Y_{u-},N_{u-},\ v)\wp(\omega)\ (du,\ dv) =X_{\Bar{\sigma}_{1}},\\
    Y_{t}(\omega)&=Y_{\Bar{\sigma}_{1}}+(t-\bar{\sigma}_{1}) - \int_{\Bar{\sigma}_{1}+}^{t}\int_{0}^{C_{X_{\Bar{\sigma}_{1}}}}(Y_{u-})g_{\Lambda}(X_{u-},\ Y_{u-},N_{u-},\ v)\wp(\omega)\ (du,\ dv)=Y_{\Bar{\sigma}_{1}}+(t-\bar{\sigma}_{1}),\\
    N_{t}(\omega)&= N_{\Bar{\sigma}_{1}}+\int_{\Bar{\sigma}_{1}+}^{t}\int_{0}^{C_{X_{\Bar{\sigma}_{1}}}}g_{\Lambda}(X_{u-},\ Y_{u-},N_{u-},\ v)\wp(\omega)(du,\ dv)=N_{\Bar{\sigma}_{1}}.
\end{align*}
Hence by using above equalities, at $t=\Bar{\sigma}_{2},$
\begin{align*}
X_{\Bar{\sigma}_{2}}(\omega)&=X_{\Bar{\sigma}_{1}} +\int_{0}^{C_{X_{\Bar{\sigma}_{1}}}}h_{\Lambda}(X_{\Bar{\sigma}_{1}},\ Y_{\Bar{\sigma}_{2}-},N_{\Bar{\sigma}_{1}},\ v)\wp(\omega)(\{\Bar{\sigma}_{2}\}\times dv),\\
Y_{\Bar{\sigma}_{2}}(\omega)&=Y_{\Bar{\sigma}_{1}}+ (\Bar{\sigma}_{2}-\bar{\sigma}_{1}) -\int_{0}^{C_{X_{\Bar{\sigma}_{1}}}}(Y_{\Bar{\sigma}_{1}-})g_{\Lambda}(X_{\Bar{\sigma}_{1}},\ Y_{\Bar{\sigma}_{2}-},N_{\Bar{\sigma}_{1}},\ v)\wp(\omega)(\{\Bar{\sigma}_{2}\}\times dv),\\
N_{\Bar{\sigma}_{2}}(\omega)&= N_{\Bar{\sigma}_{1}}+\int_{0}^{C_{X_{\Bar{\sigma}_{1}}}}g_{\Lambda}(X_{\Bar{\sigma}_{1}},\ Y_{\Bar{\sigma}_{2}-},N_{\Bar{\sigma}_{1}},\ v)\wp(\omega)(\{\Bar{\sigma}_{2}\}\times\ dv).
\end{align*}
Continuing in the similar way we can construct a solution in a unique manner for each consecutive interval $(\Bar{\sigma}_{m},\ \Bar{\sigma}_{m+1}],$  where $m\geq 2$. Moreover, for a fixed $\omega$, $X_{t}(\omega)=X_{\Bar{\sigma}_{m}(\omega)}$ for all $t\in [\Bar{\sigma}_{m}(\omega),\Bar{\sigma}_{m+1}(\omega))$. Hence $X$ is an r.c.l.l. and piece-wise constant process almost surely on $[0, \lim_{m\to \infty} \bar{\sigma}_m]$. By this, Part (1) and (2) would have followed if $\{\Bar{\sigma}_{m}\}_{m\ge 1}$ were the jump times, which may not be true. For this reason, now we select an appropriate sub-sequence of $\Bar{\sigma}_{m}$ which are the jump times. To this end we first note that $\int_{0}^ {C_{X_{t-}}} g_{\Lambda}(X_{t-},\ Y_{t-},N_{t-}, v) \wp(\omega)(\{t\}\times dv)$ is zero for all $t\in(\Bar{\sigma}_m,\Bar{\sigma}_{m+1})$ for every $m\ge 1$. Using this and  $X_{\Bar{\sigma}_{m}-}=X_{\Bar{\sigma}_{m-1}}$, $N_{\Bar{\sigma}_{m}-}=N_{\Bar{\sigma}_{m-1}}$ we introduce
\begin{align}\label{n1}
l_1:=\min\{m\ge 1\colon \int_{0}^{C_{X_{\Bar{\sigma}_{m-1}}}}g_{\Lambda}(X_{\Bar{\sigma}_{m-1}},\ Y_{{\bar\sigma}_m-},N_{{\bar\sigma}_{m-1}},\ v)\wp(\omega)(\{{\bar\sigma}_m\}\times dv) \neq 0 \}.
\end{align}
From \eqref{h} and \eqref{g} it is evident that $h_{\Lambda}$ and $g_{\Lambda}$ have identical supports in $v$ variable. Hence the integrals $\int_{I}h_{\Lambda}(X_{t-},Y_{t-},N_{t-}v)\wp(\{t\} \times dv)$ is non-zero if and only if  $\int_{I}g_{\Lambda}(X_{t-},Y_{t-},N_{t-}v)\wp(\{t\} \times dv)$ is non-zero, where $I$ is any interval. Then $t={\bar\sigma}_{l_1}$ is the first time when both the integrals
$$\int_{[0,C_{X_{t-}}]} g_{\Lambda}(X_{t-},\ Y_{t-},N_{t-},\ v)\wp(\omega)(\{t\}\times dv) \quad \textrm{ and }\quad \int_{[0,C_{X_{t-}}]}h_{\Lambda}(X_{t-},\ Y_{t-},N_{t-},\ v)\wp(\omega)(\{t\}\times dv)$$ are non-zero.
Consequently, $Y_{t}=Y_{0}+t$, $X_{t}=X_{0}$, and $N_{t}=0$, for all $t\in [0,\Bar{\sigma}_{l_1})$. Hence $X_{{\bar\sigma}_{l_1}-}=X_0$, $Y_{{\bar\sigma}_{l_1}-}=Y_0+\Bar{\sigma}_{l_1}$  and $N_{{\bar\sigma}_{l_1}-}=0$. Furthermore, at $t=\bar\sigma_{l_1}$,
\begin{align*}
0&\neq \int_{0+}^{\bar \sigma_{l_1}}\int_{0}^{C_{X_{\Bar{\sigma}_{l_{1}-1}}}}g_{\Lambda}(X_{t-},\ Y_{t-},N_{t-},\ v)\wp(\omega)(dt, dv)\\
&= \int_{0}^{C_{X_{\Bar{\sigma}_{l_{1}-1}}}}g_{\Lambda}(X_{\Bar{\sigma}_{l_{1}-1}},\ Y_{\Bar{\sigma}_{l_1}-},N_{\Bar{\sigma}_{l_1-1}},\ v)\wp(\omega)(\{\Bar{\sigma}_{l_1}\}\times dv) = 1,
\end{align*}
using \eqref{n1} and the facts that $\wp(\omega)(\{\bar \sigma_{l_1}\} \times [0,C_{X_{\Bar{\sigma}_{l_{1}-1}}}]) =1$ and $g_{\Lambda}(i,y,n,v)\in \{0, 1\}$ for every $i$, $y$, $n$, and $v$. Thus from \eqref{eq_N(t)} and above expressions
\begin{align*}
N_{\Bar{\sigma}_{l_1}} =& \int_{0+}^{\Bar{\sigma}_{l_1}}\int_{0}^{C_{X_{\Bar{\sigma}_{l_{1}-1}}}}g_{\Lambda}(X_{t-},\ Y_{t-},N_{t-},\ v)\wp(\omega)(dt,dv) = 1.
\end{align*}
Similarly, using  \eqref{eqY} we have $Y_{\Bar{\sigma}_{l_1}}=0$. To see this, we write $Y_{\Bar{\sigma}_{l_1}}= Y_{\Bar{\sigma}_{l_1-}} + (Y_{\Bar{\sigma}_{l_1}}-Y_{\Bar{\sigma}_{l_1}-})$, i.e.,
\begin{align*}
Y_{\Bar{\sigma}_{l_1}} =&Y_{\Bar{\sigma}_{l_1-}} - \int_{0}^{C_{X_{\Bar{\sigma}_{l_{1}-1}}}}(Y_{\Bar{\sigma}_{l_1}-})g_{\Lambda}(X_{\Bar{\sigma}_{l_1}-}, Y_{\Bar{\sigma}_{l_1}-},N_{\Bar{\sigma}_{l_1}-},\ v)\wp(\omega)(\{\Bar{\sigma}_{l_1}\}\times dv)\\
=& Y_{\Bar{\sigma}_{l_1} -} - Y_{\Bar{\sigma}_{l_1} -}
=0.
\end{align*}
In general, for every $n\ge 1$, we set
\begin{equation}\label{subsequence}
    l_{n+1}:=\min\left\{m> l_{n}\colon \int_{0}^{C_{X_{\Bar{\sigma}_{m-1}}}}g_{\Lambda}(X_{\Bar{\sigma}_{m-1}},\ Y_{\Bar{\sigma}_m-},N_{\Bar{\sigma}_{m-1}},\ v)\wp(\omega)(\{\Bar{\sigma}_m\}\times dv) \neq 0 \right\}.
\end{equation}
 In other words, for every $t\ge 0$,
\begin{align}\label{intg}
    \int_{\mathbb{R}}g_{\Lambda}(X_{t-},\ Y_{t-},N_{t-}\ v)\wp(\{t\}\times dv)= \left\{\begin{array}{ll}
	   1 & ,\ \mathrm{if}\ t=\Bar{\sigma}_{l_n} \textrm{ for some } n\ge 1\\
	   0 &,\ \mathrm{otherwise}.
	   \end{array}\right.
\end{align}
From \eqref{intg} and \eqref{eq_N(t)} we get
\begin{align}
     N_{t}=&\sum_{\{r\ge 1\mid \Bar{\sigma}_{l_{r}}\le t \}} 1 =\sum_{r=1}^{\infty} \mathds{1}_{[\Bar{\sigma}_{l_{r}},\infty)}(t). \label{NT}
\end{align}
Moreover, from  \eqref{intg} and \eqref{eqY} we get
\begin{align}
     Y_{t}= & Y_{0}+t- \sum_{\{r\ge 1\mid \Bar{\sigma}_{l_{r}}\le t \}}Y_{\Bar{\sigma}_{l_{r}}-}=Y_{0}+t- \sum_{r=1}^{\infty}(Y_{\Bar{\sigma}_{l_{r}}-}) \mathds{1}_{[\Bar{\sigma}_{l_{r}},\infty)}(t). \label{YT}
\end{align}
Furthermore, as the support of $g_\Lambda$ and $h_\Lambda$ are identical, from  \eqref{subsequence} we have that $\int_{\mathbb{R}}h_{\Lambda}(X_{t-},\ Y_{t-},N_{t-}\ v)\wp(\{t\}\times dv)$ is nonzero if and only if   $t=\Bar{\sigma}_{l_n}$ for some $n\ge 1$. Therefore, if $n\ge 1$ is such that $\bar\sigma_{l_n} \le t < \bar\sigma_{l_{n+1}}$, then $X_t= X_{\bar\sigma_{l_n}}$. Thus we can write
\begin{align}
X_{t}=& X_0 + \sum_{\{r\ge 1\mid \Bar{\sigma}_{l_{r}}\le t \}} (X_{\Bar{\sigma}_{l_{r}}}- X_{\Bar{\sigma}_{l_{r-1}}})  =  X_0 +  \sum_{r=1}^{\infty}(X_{\Bar{\sigma}_{l_{r}}}- X_{\Bar{\sigma}_{l_{r-1}}})  \mathds{1}_{[\Bar{\sigma}_{l_{r}},\infty)}(t).\no
\end{align}
Hence $X$ and $N$ are r.c.l.l., and piece-wise constant and $Y$ is r.c.l.l., and piece-wise linear. We denote $T_n:=\Bar{\sigma}_{l_{n}}$ for each $n\ge 1$. From above, it is evident that $\{T_n\}_{n\ge 1}$ is the desired sequence of stopping times at which the processes $X$, $Y$, and $N$ jump and the properties in parts (1) and (2) hold.

\noindent Next for proving part (3), we first rewrite \eqref{YT}
\begin{align}\no
         Y_{t}= & Y_{0}+t- \sum_{\{r\ge 1\mid T_r\le t \}}Y_{T_r-}
\end{align}
and thus for all $n\in \mathbb{N},$
\begin{align*}
    Y_{T_{n}}=& Y_{0}+T_{n}- \sum_{\{r\ge 1\mid T_r\le T_{n} \}}Y_{T_{r}-}\\
    =&  \left( Y_{0}+T_{n}- \sum_{\{r\ge 1\mid T_r< T_{n} \}}Y_{T_{r}-}\right) - Y_{T_{n}-}\\
    =&  Y_{T_{n}-}-  Y_{T_{n}-} =0.
\end{align*}
Thus (i) holds. For showing (ii) we first recall that $Y_{T_1-} = Y_0+ T_1$, which is same as $T_1-T_0$. Now for $n>1$ using $Y_{T_{n-1}}=0$ we get
\begin{align*}
 Y_{T_n-}  =&  \left( Y_{0}+T_{n}- \sum_{\{r\ge 1\mid T_r< T_{n} \}}Y_{T_{r}-}\right)\\
=&(T_n-T_{n-1}) +\left( Y_{0}+T_{n-1}- \sum_{\{r\ge 1\mid T_r\le T_{n-1} \}} Y_{T_{r}-}\right)\\
=& (T_n-T_{n-1}) + Y_{T_{n-1}}= (T_n-T_{n-1}).
\end{align*}
Hence (ii) is shown. Again, \eqref{NT} implies that $N_t= \max\{r: T_r\le t \}$. That is $N_{T_n}= \max\{r: T_r\le T_n\}=n$, which proves (iii). Finally, (iv) follows using \eqref{YT} and (ii) as below
\begin{align}\no
         Y_{t}= & Y_{0}+t- \sum_{\{1\le r\le N_t \}}(T_r-T_{r-1}) = Y_0 +t - (T_{N_t}- T_0)= t-T_{N_t}. \qed
\end{align}
\end{proof}
\begin{rem}
The jump times $\{T_{n}\}_{n\ge 1}$ of the process $X$ are called the transition times. We need to prove that this sequence diverges to infinity almost surely for establishing that the solution is globally determined with probability 1. This is accomplished in the following Theorem.
\end{rem}

\begin{theo}\label{explosion}
There exists a unique strong solution $Z=(X_t,Y_t,N_t)_{t\ge 0}$ of the coupled SDE \eqref{eqX}-\eqref{eq_N(t)}
\end{theo}
\begin{proof}
Let $\{T_{n}\}_{n\ge 1}$ be as in Theorem \ref{theo3.1}.
Now we will show that $T_n$ diverges. Let $\tau:= \lim_{n\to \infty} T_n$.
Clearly,  for a fixed $\epsilon >0$, the event $\{\tau<\infty\}$ is a subset of $\cup_{n_0\ge 1}\{T_n-T_{n-1}<\epsilon,\;\forall n\ge n_0\}$. For the sake of brevity, we denote $\{T_n-T_{n-1}<\epsilon\}$ as $A_n$ for each $n$. Therefore, we will  show $P(\cap_{n\ge n_0} A_n)=0$ for each $n_0$ which is enough for proving  $P( \tau<\infty)=0$.
To this end we first compute the conditional probability of the event $A_n$ given the observations till $T_{n-1}$, using the properties of Poisson random measure. Indeed, using part (3) of Theorem \ref{theo3.1}, occurrence of $A_n$ is same as having Poisson random measure of $ \cup_{0<y<\epsilon}\left(\{ T_{n-1} +y\}\times\Lambda_{X_{ T_{n-1} }}(y,n-1)\right)$ positive.
Thus using the Lebesgue intensity of $\wp$, we have
    \begin{align}\label{inequlity1}
    &P\left(A_n\mid T_{n-1}, X_{T_{n-1}}, \ldots, T_0,X_0\right)\no\\
    &=P\left(\wp\left(\underset{0<y<\epsilon}{\cup}\left(\{ T_{n-1} +y\}\times\Lambda_{X_{ T_{n-1} }}(y,n-1)\right)\right)\neq0\mid T_{n-1}, X_{T_{n-1}}, \ldots, T_0, X_0\right)\no\\
    &=\left(1-e^{-\int_0^\epsilon \lambda_{X_{T_{n-1} }}(y,n-1)dy }\right)\no\\
        &\le \left(1-e^{-\epsilon c}\right)
\end{align}
which is a deterministic constant. We recall that the last inequality is due to Assumption (A1).
Next we note that if
\begin{align}\label{Ineq}
E\left[\prod_{n_0\le n\le l}  \mathds{1}_{A_n}\right] \le \left(1-e^{-\epsilon c}\right) E\left[\prod_{n_0\le n\le l-1} \mathds{1}_{A_n}\right]
\end{align} holds for all $l\ge n_0$, using that repeatedly, we get
\begin{align*}
P\left(\cap_{ n_0\le n} A_n\right)&\le P\left(\cap_{n_0\le n\le l}A_n\right)  = E\left[\prod_{n_0\le n\le l} \mathds{1}_{A_n}\right] \le (1-e^{-\epsilon c})^{l-n_0+1}
\end{align*}
for all $l\ge n_0$. The right side clearly vanishes as $l\to \infty$, and thus $P(\cap_{ n_0\le n} A_n)=0$, as desired, provided \eqref{Ineq} holds.
Finally we show \eqref{Ineq} below using the property of conditional expectation, and inequality \eqref{inequlity1}
\begin{align*}\label{probability}
E\left[\prod_{n_0\le n\le l} \mathds{1}_{A_n}\right]
      & = E\left[E\left[\prod_{n_0\le n\le l} \mathds{1}_{A_n}\Big |\; T_{l-1}, X_{T_{l-1}}, \ldots, T_0,X_0\right]\right] \\
      &=E\left[\left(\prod_{n_0\le n\le l-1} \mathds{1}_{A_n}\right)E\left[\mathds{1}_{A_l}\Big | \;T_{l-1}, X_{T_{l-1}}, \ldots, T_0,X_0\right]\right]\\
      &\le \left(1-e^{-\epsilon c}\right) E\left[\prod_{n_0\le n\le l-1} \mathds{1}_{A_n}\right]
    \end{align*}
for all $l\ge n_0$. Hence the proof is complete.
\qed
\end{proof}
\noindent The above theorem essentially asserts that the jump process $X$ is pure. In the next section we show that $X$ is a semi-Markov process. \section{Law of the Solution}
\subsection{Semi-Markovity}
\begin{theo}\label{theo3.2}
Let $Z= (X,Y,N)=\{(X_{t},\ Y_{t},N_{t})\}_{t\geq 0}$ be the unique strong solution to  \eqref{eqX}-\eqref{eq_N(t)}. Then the following hold.
\begin{enumerate}[i.]
\item The process $\{X_{t}\}_{t\geq0}$ is a pure SMP.
\item The embedded chain $\{X_{T_n}\}_{n\ge 1}$ is a discrete time Markov chain.
\end{enumerate}
\end{theo}
\begin{proof}
First we will prove that $X$ is SMP. We have already seen in the proof of Theorem \ref{theo3.1}, that $X$ is an r.c.l.l. process. Next, we need to show \eqref{defi_semimarkov}, i.e., for each $n \in \mathbb{N}$, $y\in \mathbb{R}_0$, $j\in \mathcal{X}$
   \begin{equation}\label{semi}
       P[X_{T_{n+1}}=j,T_{n+1}-T_{n}\leq y\mid X_{0},T_{0}, X_{T_{1}},T_{1} ,\ldots, X_{T_{n}},T_{n}]=P[X_{T_{n+1}}=j,T_{n+1}-T_{n}\leq y \mid X_{T_{n}}].
   \end{equation}	
We note that the left side is equal to
\begin{eqnarray} \label{note 2}
\no && P(T_{n+1}-T_n\leq y\mid (X_0,{T_0}),(X_{T_1},{T_1}),\ldots,(X_{T_n},{T_n}))\\
	&& \times P(X_{T_{n+1}}=j\mid (X_0,{T_0}),(X_{T_1},{T_1}),\ldots,(X_{T_n},{T_n}),\{ T_{n+1}-T_n\leq y \}).
\end{eqnarray}

\noi Each of the two conditional probabilities is further simplified below. For almost every $ \omega\in\Omega $, Equation \eqref{intg} and Theorem \ref{theo3.1} part (3)(iii)-(iv) imply that for any $n\ge 1$
\begin{equation*}
	\int_{T_n+}^{T_n+t}(u-T_n)\int_{\mathbb{R}}g_{\Lambda}(X_{T_n},u-T_n,n,v)\,\wp(du,dv)=\begin{cases}
	0, &\text{ for } t<T_{n+1}-T_n\\
	T_{n+1}-T_n, & \text{ for } t= T_{n+1}-T_n.
	\end{cases}
\end{equation*}
\noi Hence, by a suitable change of variable, almost surely
$T_{n+1}-T_n $ is the first occurrence of a non-zero value of the following map
\begin{equation*}
	t\mapsto\int_{0+}^{t} u \int_{\mathbb{R}}g_{\Lambda}(X_{T_n},u,n,v)\,\wp(T_n+du,dv)
\end{equation*}
and that occurs at $t=T_{n+1}-T_n$. Again, since $ \wp(T_n+du,dv) $ is independent to $ \mathcal{F}_{T_n}$ we obtain, $T_{n+1}-T_n$ is conditionally independent to $ \mathcal{F}_{T_n}$ given $X_{T_n}$. Thus for all $n\ge 1$
\begin{eqnarray}\label{note 4}
&&P\left(T_{n+1}-T_n\leq y\mid (X_0,{T_0}),(X_{T_1},{T_1}),\dots,(X_{T_n},{T_n})\right) = P(T_{n+1}-T_n\leq y\mid X_{T_n}).
\end{eqnarray}
\noi By substituting $t$ equal to $T_{n}$ and $T_{n+1}$ in Equation \eqref{eqX}, and using Theorem \ref{theo3.1} part (3) (ii)-(iii) we get
\begin{align}\label{note Xt}
	X_{T_{n+1}}=X_{T_n}+\int_{\mathbb{R}}h_{\Lambda}(X_{T_n},T_{n+1}-T_n,n,v)\,\wp(\{T_{n+1}\}\times dv),
\end{align}
as $X_{T_{n+1}-}=X_{T_{n}}$, $ Y_{T_{n+1}-}=T_{n+1}-T_n $ and $N_{T_{n+1}-}=n$. Thus using \eqref{note Xt}
\begin{align*}
P\left(X_{T_{n+1}}=j\mid (X_0,{T_0}),(X_{T_1},{T_1}),\ldots,(X_{T_n},{T_n}),\{ T_{n+1}-T_n\leq y \}\right)&\\
	= \no P\bigg( \int_{\mathbb{R}}h_{\Lambda}(X_{T_n},T_{n+1}-T_n,n,v)\,\wp(\{T_n+(T_{n+1}-T_n)\}\times dv) =&j-X_{T_n}\Big|\\
	 \no (X_0,{T_0}),(X_{T_1},{T_1}),\ldots,(X_{T_n},{T_n}),&\{ T_{n+1}-T_n\leq y \}   \bigg).
\end{align*}
Again, using the independence of $\wp(T_n+du,dv)$ to $ \mathcal{F}_{T_n}$ and conditional independence of $T_{n+1}-T_n$ to $\mathcal{F}_{T_n}$ given $X_{T_n}$ we conclude, the right side expression is equal to
\begin{align}\label{note 3a}
	& \no P\left(\int_{\mathbb{R}}h_{\Lambda}(X_{T_n},T_{n+1}-T_n,n,v)\,\wp(\{T_n+(T_{n+1}-T_n)\}\times dv)=j-X_{T_n} ~\Big|~ X_{T_n}, \{ T_{n+1}-T_n\leq y \}  \right)
\end{align}
which is again equal to $ P\left( X_{T_{n+1}}=j\mid X_{T_n},\{ T_{n+1}-T_n\leq y \} \right)$ using \eqref{note Xt}. Thus, using this simplification and \eqref{note 4} in \eqref{note 2}, we obtain for all $n\ge 1$
\begin{align*}
	&P(X_{T_{n+1}}=j,T_{n+1}-T_n\leq y\mid (X_0,{T_0}),(X_{T_1},{T_1}),\ldots,(X_{T_n},{T_n}))\\
	&= P\left(T_{n+1}-T_n\leq y\mid X_{T_n}\right) P\left(X_{T_{n+1}}=j\mid X_{T_n},\{ T_{n+1}-T_n\leq y \}\right)\\
	&= P\left(X_{T_{n+1}}=j,T_{n+1}-T_n\leq y\mid X_{T_n} \right).
 \end{align*}
\noi Hence, part $(i)$ is proved. \\
For every $n\ge 1$, taking $y\to \infty$ on both sides in Equation \eqref{semi}, we get
$$P[X_{T_{n+1}}=j\mid X_{T_{0}},T_{0}, X_{T_{1}},T_{1} ,\ldots, X_{T_{n}},T_{n}]=P[X_{T_{n+1}}=j \mid X_{T_{n}}].$$
Hence the part $(ii)$.
\qed
\end{proof}

\subsection{Transition Parameters}
\noi For each $i \in \mathcal{X}, n\in \mathbb{N}_{0}$, we define a function $F(\cdot\mid i,n)\colon [0,\infty)\rightarrow [0,1]$ as
   	\begin{equation}\label{distribution}
   	    F(y\mid i,n):=1-e^{-\gamma_{i}(y,n)}
   	\end{equation}
where $\gamma_{i}(y,n)$ is as in (A2).
Clearly, $F(\cdot\mid i,n)$ is differentiable almost everywhere. To see this, we note that $\gamma_{i}(y,n)$ is an integral of a bounded Lebesgue measurable function, and thus is absolutely continuous in $y$. Let $ f(y\mid i,n)$ be the almost everywhere derivative of $F(y\mid i,n)$. We also define for each $y\in \mathbb{R}_{+}$, $n\in \mathbb{N}_0$, a matrix $p(y,n):=(p_{ij}(y,n))_{\mathcal{X}\times \mathcal{X}}$, such that
	\begin{equation}\label{p_ij}
		p_{ij}(y,n):=\begin{cases}
		\dfrac{\lambda_{ij}(y,n)}{\lambda_{i}(y,n)}\mathds{1}_{(0,\infty)}(\lambda_{i}(y,n)),&\textrm{ if } j\neq i\\
		\mathds{1}_{\{0\}}(\lambda_{i}(y,n)), 	&\textrm{ if }  j=i.
	\end{cases}
	\end{equation}
Since, $\lambda_{ij}(y,n)\ge 0$ and $\sum_{j\in \mathcal{X} \setminus\{i\}} \lambda_{ij}(y,n) = \lambda_{i}(y,n)$, for each $n\in \mathbb{N}_0$ and $y\in \mathbb{R}_{+}$, $p(y,n)$ is a transition probability matrix.
The following proposition asserts that $p(y,n)$ gives the conditional probability of selecting a state at the time of $n+1$\textsuperscript{st} transition given the age $y$ and location $i$ of the previous state. Furthermore, the map $F(\cdot\mid i,n)$ as in \eqref{distribution} is also asserted as the conditional cumulative distribution function of the holding time at the $n$th state given that is $i$.
\begin{theo}\label{prop3.5} Let $Z=(X,Y,N)$ be the solution to \eqref{eqX}-\eqref{eq_N(t)}, $i\in \mathcal{X}$, $y\in \mathbb{R}_{+}$ and $n\ge 1$ then the following hold.
\begin{enumerate}[i.]
\item $F(\cdot\mid i,n)$ is the conditional cumulative distribution function of the holding time of the process $X$.
\item For all $j\neq X_{T_n}$, $p_{X_{T_n}j}(T_{n+1}-T_{n},n)=P[X_{T_{n+1}}=j\mid X_{T_n}, T_{n+1}-T_{n}]$ almost surely.
\end{enumerate}
	\end{theo}

\begin{proof}
We recall that, the intensity of $\wp$ is Lebesgue measure, and the Lebesgue measure of $\{(u,v)\in (T_n,T_n+y) \times \mathbb{R}_+\mid v\in \Lambda_{i}(u-T_{n},n)\}$ is $\int_{T_n}^{T_n+ y}\lambda_{i}(u-T_n,n) du$ which is equal to $\gamma_{i}(y,n)$ (see (A2)). Using \eqref{g} and \eqref{intg}, the conditional probability of no transition in the next $y$ unit time, given that the $n\textsuperscript{th}$ transition happens now to state $i,$ is given by $e-{}\gamma_i(y,n)$
	Using the above fact, the conditional cumulative distribution function at $y$ of the holding time after the $n$\textsuperscript{th} transition, given the state is $i$, is	
	\begin{align}
		    P[T_{n+1}-T_{n}\leq y\mid X_{T_n}=i]&=1-P[X_{t}=X_{t-},\forall t\in(T_n,T_n+y]\mid X_{T_n}=i]\nonumber\\
		    &= 1-e^{-\gamma_{i}(y,n)}\nonumber
		\end{align}
		for all $y\in \mathbb{R}_{+}$ and $i\in \mathcal{X}.$ Thus (i) follows from \eqref{distribution}.

\noi We note that, for $j\neq i$, $ P[X_{T_{n+1}}=j\mid X_{T_n}=i, T_{n+1}-T_{n}=y] $ is the conditional probability of the event that the $n+1$\textsuperscript{st} state is $j$, given that $T_{n+1}=T_n +y$ and the $n$th state is $i$. Using \eqref{note Xt}, the above is the conditional probability that a Poisson point mass appears in $\{T_{n} +y\}\times\Lambda_{ij}(y,n)$ given that the point mass lies somewhere in $\{T_{n} +y\}\times \Lambda_{i}(y,n)$ and no transition of $X$ occurs during $(T_{n},T_{n}+y)$. If these three events are denoted by $A$, $B$, and $C$ respectively, then the conditional probability $P(A\mid B\cap C)$ can be simplified as $P(A\mid B)$ because $C$ is independent to both $A$ and $B$. Thus using the Lebesgue intensity of $\wp$,
\begin{align*}
P\left[X_{T_{n+1}}=j\mid X_{T_n}=i, T_{n+1}-T_{n}=y\right]=& P\left[\wp(\{T_{n} +y\}\times\Lambda_{ij}(y,n))=1 \mid \wp(\{T_{n} +y\}\times \Lambda_{i}(y,n))=1\right]\\
=&\dfrac{|\Lambda_{ij}(y,n)|}{|\Lambda_{i}(y,n)|}
\\
=& \dfrac{\lambda_{ij}(y,n)}{\lambda_{i}(y,n)}
\end{align*} for every $y\in \mathbb{R}_{+}$, and $j\neq i$, provided $\lambda_{i}(y,n)\neq 0$. Thus from \eqref{p_ij}, we get $ p_{ij}(y,n)$ is equal to\\ $P\left[X_{T_{n+1}}=j\mid X_{T_n}=i, T_{n+1}-T_{n}=y\right] $ when $\lambda_{i}(y,n)\neq 0$. Next we note that $\lambda_{i}(y,n)= 0$ if and only if \\$\frac{d}{dy} F(y\mid i,n) =0$, i.e., the density of $T_{n+1}-T_n$ is zero at $y$. Hence (ii) holds.
\qed
\end{proof}
\begin{rem}
\noi Using \eqref{distribution}, we note that under Assumptions (A1) and (A2), $ F(y\mid i,n)<1 $ for all $y\in \mathbb{R}_{+}$ and $ {\lim}_{y\rightarrow\infty} F(y\mid i,n)=1$ using \eqref{distribution}. Thus, the holding times are unbounded but finite almost surely. By dropping (A1), one may include a class of SMPs having bounded holding times. However, we exclude that class from our discussion. It is also important to note that the SMPs having discontinuous cdf of holding times are also not considered in the present setting. Nevertheless, the present study subsumes countable-state continuous time Markov chains and the processes having age dependent transitions as appear in \cite{GhoshSaha2011, tanmay}.
\end{rem}

\begin{prop}\label{prop3.6}
We have, for almost every $ y\geq 0 $ and $n\ge 1$,
\begin{equation*}
p_{ij}(y,n)\dfrac{f(y\mid i,n)}{1-F(y\mid i,n)}=\begin{cases}
\lambda_{ij}(y,n),&\textrm{ for } i\neq j,\\
0,&\textrm{ for } i=j.
		\end{cases}
		\end{equation*}
\end{prop}
\begin{proof}
By differentiating both sides of  \eqref{distribution}, we obtain $f(y\mid i,n)=\lambda_{i}(y,n)e^{-\gamma_i(y,n)}$ for every $y\in \mathbb{R}_{+}$. This is equal to $\lambda_{i}(y,n)(1-F(y\mid i,n))$ using \eqref{distribution}. Hence, for every $y\in \mathbb{R}_{+}$, $n\ge 1$ and $i\in \mathcal{X}$
\begin{align}
\frac{f(y\mid i,n)}{1-F(y\mid i,n)}&=\lambda_{i}(y,n)\label{lamb ii}.
\end{align}
If $ i\neq j $, for every $y\in \mathbb{R}_{+}$, using \eqref{p_ij}
\begin{align}
\no p_{ij}(y,n)\dfrac{f(y\mid i,n)}{1-F(y\mid i,n)} =& \dfrac{\lambda_{ij}(y,n)}{\lambda_{i}(y,n)}\times\lambda_{i}(y,n) \mathds{1}_{(0,\infty)}(\lambda_{i}(y,n))
= \lambda_{ij}(y,n)
\end{align}
as $0\le \lambda_{ij}(y,n) \le \lambda_{i}(y,n)$.
The case for $i=j$ follows from \eqref{p_ij} and \eqref{lamb ii} directly as $p_{ii}(y,n)\dfrac{f(y\mid i,n)}{1-F(y\mid i,n)}$ is equal to $\lambda_{i}(y,n) \mathds{1}_{\{0\}}(\lambda_{i}(y,n))$ which is zero. \qed
\end{proof}

\begin{theo} \label{theoQ}
		Let $X$ be a SMP as in Theorem \ref{theo3.2}. Then, the associated kernel is given by
	\[P[X_{T_{n+1}}=j, T_{n+1}-T_{n}\leq y \mid  X_{T_n}=i]=\int_{0}^{y}e^{-\gamma_{i}(s,n)}\lambda_{ij}(s,n)\ ds,\] which is denoted by $Q_{ij}(y,n)$ for  every $y>0$, $n\ge 1$ and $i\neq j.$
\end{theo}
	\begin{proof}
		Using Theorem \ref{prop3.5} (i) and (ii) and Theorem \ref{theo3.1} part (3) (ii)
\begin{align*}
&P\left[X_{T_{n+1}}=j, T_{n+1}-T_{n}\leq y \mid  X_{T_n}=i\right]\\
&=\mathbb{E}\left[P\left(X_{T_{n+1}}=j, T_{n+1}-T_{n}\leq y\mid X_{T_n}=i,T_{n+1}-T_{n}\right)\mid X_{T_n}=i\right]\\
&=\int_{0}^{\infty}\mathds{1}_{[0,y]}(s)P\left[X_{T_{n+1}}=j\mid X_{T_n}=i,T_{n+1}-T_{n} =s\right]f(s\mid i,n)\,ds\\
&=\int_0^y p_{ij}(s,n)f(s\mid i,n)\,ds.
\end{align*}
For each $i\neq j$, using Proposition \ref{prop3.6} and \eqref{distribution}, the right side of above can be rewritten as
\begin{align*}
&\int_0^y (1-F(s\mid i,n))\lambda_{ij}(s,n)\,ds\ =\int_{0}^{y}e^{-\gamma_{i}(s,n)}\lambda_{ij}(s,n)\,ds = Q_{ij}(y,n).
\end{align*}
\qed
\end{proof}
\begin{prop}
Let $X$ be a SMP as in Theorem \ref{theo3.2}. Then, $\lambda(n,y)$ is the instantaneous transition rate matrix.
\end{prop}
\proof
The rate of transition from state $i$ to $j$ at age $y$ is given by
\begin{align*}
    &\lim_{h\to 0}\frac{1}{h}\left[P(X_{T_{n+1}}=j,T_{n+1}-T_n\in(y,y+h]\mid X_{T_{n}}=i,\{T_{n+1}-T_n>y\})\right]\\
    & =\lim_{h\to 0}\frac{1}{h}\frac{P(X_{T_{n+1}}=j,T_{n+1}-T_n\in(y,y+h]\mid X_{T_{n}}=i)}{P(T_{n+1}-T_n>y\mid X_{T_n}=i)}\\
    & =\lim_{h\to 0}\frac{1}{h}\frac{P(X_{T_{n+1}}=j,T_{n+1}-T_n \le y+h\mid X_{T_{n}}=i)- P(X_{T_{n+1}}=j,T_{n+1}-T_n \le y\mid X_{T_{n}}=i) }{1- P(T_{n+1}-T_n \le y\mid X_{T_n}=i)}.
\end{align*}
Using Theorem \ref{theoQ}, the above limit is equal to $\frac{\frac{d}{dy}Q_{ij}(y,n)}{1-F(y\mid i,n)}$ which can further be simplified as $\lambda_{ij}(y,n)$. \qed
\begin{rem}
We have obtained $Q_{ij}(y,n)=\int_0^y p_{ij}(s,n)f(s\mid i,n)\,ds$ in the proof of Theorem \ref{theoQ}, which expresses $Q_{ij}(\cdot,n)$ in terms of the $p_{ij}(\cdot,n)$, and $f(\cdot\mid i,n)$. These parameters give the age dependent transition probabilities and the conditional holding time densities. In an alternative conditioning, the kernel can also be expressed as $Q_{ij}(y,n) = P\left[X_{T_{n+1}}=j \mid  X_{T_n}=i\right] P\left[ T_{n+1}-T_n \leq y \mid  X_{T_n}=i, X_{T_{n+1}}=j \right]$, which is the product of transition probabilities of embedded chain and the conditional cdf of holding time given the current and the next states. Generally an SMP is characterized using the transition kernel $Q$. Although, instantaneous transition rate matrix $\lambda$ also characterizes an SMP, $Q$ is considered more fundamental as $\lambda$ exists only if $Q$ is differentiable. In that case, each of $Q$ and $\lambda$ can be expressed in terms of another, which is evident from the above two results.
\end{rem}

\section{Component-wise Semi-Markov Process}
\begin{defi}
A pure jump process $X$ on a countable state space $\mathcal{X}$ is called a component-wise Semi-Markov Process (CSM) if there is a bijection $\Gamma : \mathcal{X}\to \prod_{i=1}^{d} \mathcal{X}_{i}$, such that each component of $\Gamma(X)$ is a semi-Markov process, where $d$ is a positive integer and for each $i\le d$, $\mathcal{X}_{i}$ is an at-most countable non-empty set.
\end{defi}
\noindent Without loss of generality, we assume that $\mathcal{X}= \prod_{i=1}^d \mathcal{X}_i$ and $X^i$ is a semi-Markov process on $\mathcal{X}_i$  for each $i\le d$. Here $\Gamma$ is the identity map. Next we consider a specific CSM process of dimension $2$.
\begin{notation}\label{N4.1}
Fix $i,j\in \mathcal{X}$ and $y_1,y_2\geq 0$. Let $Z^1=(X^1,Y^1,N^1)$ and $Z^2=(X^2,Y^2,N^2)$ be the strong solutions of \eqref{eqX}-\eqref{eq_N(t)} with two different initial conditions.
At a fixed time $s(>0)$ we denote
\begin{equation}
\label{initialcondition1}
    i=X^1_s,y_1=Y^1_s,n_1=N^1_s
\end{equation}
 and
\begin{equation}
\label{initialcondition2}
     j=X_s^2,y_2=Y_s^2,n_2=N^2_s
\end{equation}
respectively. We also denote $(Z^1,Z^2)$ as $Z$.
\end{notation}
\noindent Throughout this part, we will use the notation described above. For the sake of computing some specific parameters connected to the law of Z, we limit ourselves to the following choice of $\tilde{\lambda}$.
\begin{enumerate}
\item[{\bf (A3)}.] For all $y\in \mathbb{R}_{+}$, $n\in\mathbb{N}_{0}$ and $(i,j) \in \mathcal{X}_2$, $\tilde{\lambda}_{ij}(y,n)=\max(\|\lambda_{ij}(\cdot,\cdot)\|_{L^{\infty}_{(\mathbb{R}_+\times\mathbb{N}_0)}}, \lambda_{ij}(y,n))$.
\end{enumerate}

\noindent Since, $Z^1$ and $Z^2$ as in Notation \ref{N4.1} are Markov, $Z=(Z^1,Z^2)$ is also Markov. It has state, component-wise age and component-wise number of transitions $X=(X^1, X^2)$, $Y=(Y^1,Y^2)$ and $N=(N^1,N^2)$ respectively. While each of $X^1$ and $X^2$ is semi-Markov, the pure jump process $X$ is not. Rather, $X$ is a component-wise semi-Markov process (CSM) and the Markov process $Z$ is called the augmented process of CSM $X$. Since, for our case, the components of the CSM $X$ are driven by a single Poisson random measure, they are not independent. In view of this, it is interesting to derive the law of $X$ by finding the the infinitesimal generator $\mathcal{A}$ of the augmented process $Z=(Z^1, Z^2)$. Under (A3), $\mathcal{A}$ is obtained below using It\^{o}'s lemma for r.c.l.l. semi-martingales. Let $\varphi\colon (\mathcal{X}\times \mathbb{R}_+\times\mathbb{N}_{0})^2\to \mathbb{R}$ be bounded and continuously differentiable in its continuous variables. Then using \eqref{vectorform}
\begin{align*}
&d\varphi(Z^1_t,Z^2_t)- \left(\frac{\partial}{\partial y_1}+\frac{\partial}{\partial y_2}\right)\varphi(Z^1_t,Z^2_t)dt\\
&=\varphi(Z^1_t,Z^2_t)-\varphi(Z^1_{t-},Z^2_{t-})\\
&=\varphi\left(Z^1_{t-}+\int_{\mathbb{R}_{+}}J(Z^1_{t-},v)\wp(dt,dv),Z^2_{t-}+\int_{\mathbb{R}_{+}}J(Z^2_{t-},v)\wp(dt,dv)\right)-\varphi(Z^1_{t-},Z^2_{t-})\\
&=\int_{\mathbb{R}_+}\left[\varphi(Z^1_{t-}+J(Z^1_{t-},v),Z^2_{t-}+J(Z^2_{t-},v))-\varphi(Z^1_{t-},Z^2_{t-})\right]\wp(dt,dv)\\
&= \left(\int_{\mathbb{R}_+}\left[\varphi(Z^1_{t-}+J(Z^1_{t-},v),Z^2_{t-}+J(Z^2_{t-},v))-\varphi(Z^1_{t-},Z^2_{t-})\right]dv\right)dt+dM_t
    \end{align*}
where $M$ is the martingale obtained by integration wrt the compensated Poisson random measure $\wp(dt,dv) -dtdv$.
For simplifying the above integral term, we impose (A3) and divide the derivation in two complementary cases.

\noindent Case 1: Assume $X^1_{t-}\neq X^2_{t-}$. Now under (A3), the intervals $\Lambda_{X^1_{t-}j_{1}}(Y^1_{t-},N^1_{t-})$ and $\Lambda_{X^2_{t-}j_{2}}(Y^2_{t-},N^2_{t-})$ are disjoint for every $j_1\in \mathcal{X}\setminus\{X^1_{t-}\},j_2\in \mathcal{X}\setminus\{X^2_{t-}\}$ and $N^1_{t-},N^2_{t-}$. Thus by considering these intervals where the integrand is non-zero constants, we get
\begin{align*}
&\int_{\mathbb{R}_+}\left[\varphi(Z^1_{t-}+J(Z^1_{t-},v),Z^2_{t-}+J(Z^2_{t-},v))-\varphi(Z^1_{t-},Z^2_{t-})\right]dv\\
&=\int_{\overset{2}{\underset{k=1}{\bigcup}}\left(\underset{j\in \mathcal{X}\setminus\{X^k_{t-}\}}{\cup}\Lambda_{X^k_{t-}j}(Y^k_{t-},N^k_{t-})\right)}\left[\varphi(Z^1_{t-}+J(Z^1_{t-},v),Z^2_{t-}+J(Z^2_{t-},v))-\varphi(Z^1_{t-},Z^2_{t-})\right]dv\\
&=\underset{j\in \mathcal{X}\setminus\{X^1_{t-}\}}{\sum} \left[\varphi(j,0,N^1_{t-} +1,Z^2_{t-}) -\varphi(Z^1_{t-},Z^2_{t-})\right] |\Lambda_{X^1_{t-}j}(Y^1_{t-},N^1_{t-})|
\\&\;\;\;\;\;+\underset{j\in \mathcal{X}\setminus\{X^2_{t-}\}}{\sum}\left[\varphi(Z^1_{t-},j,0,N^2_{t-} +1) -\varphi(Z^1_{t-},Z^2_{t-})\right]|\Lambda_{X^2_{t-}j}(Y^2_{t-},N^2_{t-})|
\end{align*}
where $|I|$ is the length of the interval $I$.

\noindent Case 2: Assume that $X^1_{t-} = X^2_{t-}=i$ say. Also recall that under (A3), the intervals  $\Lambda_{ij}(y_1,n_1)$ and $\Lambda_{ij}(y_2,n_2)$ are having identical left end points for almost every $y_1,y_2\geq 0$ and $n_1,n_2\in \mathbb{N}_0$. So, $\Lambda_{X^1_{t-}j_{1}}(Y^1_{t-},N^1_{t-})$ and $\Lambda_{X^2_{t-}j_{2}}(Y^2_{t-},N^2_{t-})$ are not disjoint when $j_1= j_2$. Therefore,
\begin{align*}
&\int_{\mathbb{R}_+}\left[\varphi(Z^1_{t-}+J(Z^1_{t-},v),Z^2_{t-}+J(Z^2_{t-},v))-\varphi(Z^1_{t-},Z^2_{t-})\right]dv\\
& = \int_{\underset{\substack{j\in \mathcal{X}\setminus\{i\}}}{\cup}\left(\Lambda_{ij}(Y^1_{t-},N^1_{t-})\cup\Lambda_{{ij}}(Y^2_{t-},N^2_{t-})\right)}[\varphi(Z^1_{t-}+J(Z^1_{t-},v),Z^2_{t-}+ J(Z^2_{t-},v))-\varphi(Z^1_{t-},Z^2_{t-})]dv\\
&= \underset{j\in \mathcal{X}\setminus\{i\}}{\sum} [\varphi(j,0,N^1_{t-} +1,Z^2_{t-})-\varphi(Z^1_{t-},Z^2_{t-})]|\Lambda_{ij}(Y^1_{t-},N^1_{t-})\setminus\Lambda_{ij}(Y^2_{t-},N^2_{t-})|\\
&\;\; \;\;  +\underset{j\in \mathcal{X}\setminus\{i\}}{\sum} [\varphi(Z^1_{t-},j,0,N^2_{t-} +1)-\varphi(Z^1_{t-},Z^2_{t-})]|\Lambda_{ij}(Y^2_{t-},N^2_{t-})\setminus \Lambda_{ij}(Y^1_{t-},N^1_{t-})|\\
&\;\; \;\;  +\underset{j\in \mathcal{X}\setminus\{i\}}{\sum}[\varphi(j,0,N^1_{t-} +1,j,0,N^2_{t-} +1,)-\varphi(Z^1_{t-},Z^2_{t-})]|\Lambda_{ij}(Y^1_{t-},N^1_{t-})\cap\Lambda_{ij}(Y^2_{t-},N^2_{t-})|.
\end{align*}
Hence by combining the expressions under both the cases,
\begin{align} \label{A}
\nonumber \mathcal{A}\varphi(z_1,z_2)
&:=  \left(\frac{\partial}{\partial y_1}+\frac{\partial}{\partial y_2}\right)\varphi(z_1,z_2) + \int_{\mathbb{R}_+}\left[\varphi(z_1+J(z_1,v),z_2+J(z_2,v))-\varphi(z_1,z_2)\right]dv \\
\nonumber &= \left(\frac{\partial}{\partial y_1}+\frac{\partial}{\partial y_2}\right)\varphi(z_1,z_2) + \underset{j\in \mathcal{X}\setminus\{i_1\}}{\sum}(\lambda_{i_1j}(y_1,n_1)-\delta_{i_1i_2}\lambda_{i_2j}(y_2,n_2))^{+}[\varphi(j,0,n_1+1,z_2)-\varphi(z_1,z_2)]\\
\nonumber &\;\; \;\; +\underset{j\in \mathcal{X}\setminus\{i_2\}}{\sum}(\lambda_{i_2j}(y_2,n_2)-\delta_{i_1,i_2}\lambda_{i_1j}(y_1,n_1))^+ [\varphi(z_1,j,0,n_2+1)-\varphi(z_1,z_2)]\\ &\;\; \;\;+\delta_{i_1,i_2}\underset{j\in \mathcal{X}\setminus\{i_1,i_2\}}{\sum}(\lambda_{i_1j}(y_1,n_1)\wedge\lambda_{i_2j}(y_2,n_2))[\varphi(j,0,n_1+1,j,0,n_2+1)-\varphi(z_1,z_2)]
\end{align}
where $z_1=(i_1,y_1,n_1)$, $z_2=(i_2,y_2,n_2)$, $\delta_{ij}$ is Kronecker delta, $a^+=\max(0,a)$ and $a\wedge b =\min(a,b)$. Thus we have proved the following theorem.
\begin{theo}
Under (A3), the infinitesimal generator $\mathcal{A}$ of the augmented process $Z=(Z^1, Z^2)$ is given by \eqref{A} where $Z^1=(X^1,Y^1,N^1)$ and $Z^2=(X^2,Y^2,N^2)$ are as in Notation \ref{N4.1}.
\end{theo}
\section{Conclusion}
We study strong solutions of a two-dimensional jump SDE driven by a Poisson random measure with Lebesgue intensity. The coefficients are chosen depending on a given transition rate function and an additional gaping parameter. Since, the coefficients are not compactly supported and also the intensity measure is not finite, the existence result is not straightforward. We have first proved the local existence and then established the global existence of a unique strong solution.
We then show that the state component of the solution is a pure semi-Markov process with the given transition rate function and the other component is the age process. Although the law of a single solution does not depend on the additional gaping parameter, the joint distribution of a couple of solutions with different initial conditions does depend on that. Under a simplified assumption on the gaping parameter, we derive the law of the bivariate process by calculating the infinitesimal generator of the augmented process.

To the best of our knowledge, the SDE under consideration or its any generalizations have not been studied in the literature. The approach of the proof of existence is significantly original. The SDE also gives a semi-Martingale representation of a semi-Markov process which need not be time-homogeneous. This representation is also not present in the literature. The detailed proof of the fact that the solution gives a semi-Markov process with desired transition kernel is valuable for all future study of this representation. This is a vital contribution of this paper. Finally, the semimartingale representation has been used to generate a correlated semi-Markov system with multiple members.

 In veiw of the immense applicability of the  correlated semi-Markov system, the formulation and the results presented in this paper are important. It is evident that SDE \eqref{eqX}-\eqref{eq_N(t)} generates a semi-Markov flow. Using the explicit construction of the solution of the SDE, the questions related to the meeting and merging events of multiple particles of a semi-Markov flow can be addressed in future studies.

 \section*{Acknowledgement}
 \noindent The authors are grateful to Tanmay Patankar for some helpful discussions.

\appendix
\section{}
\begin{lem}\label{lemma1}
For each fixed $\omega\in\Omega,$ consider the set $\mathcal{D}:=\{s'\in\ (0,\ \infty)\mid \wp(\omega)(\{s'\} \times E)>0\},$ where $E \in \mathcal{B}(\mathbb{R}),$ and $\wp$ is a Poisson random measure with intensity $m_{2}.$ If $m_{1}(E)<\infty,$ then set $\mathcal{D}$ has no limit point in $\mathbb{R}$ almost surely.
\end{lem}
\begin{proof}
If $m_{1}(E)<\infty$, for any natural number $n$, $\wp(\omega)([0,n]\times E)$ is a Poisson random variable with mean $n \times m_{1}(E)$. Hence $P(\wp([0,n]\times E)<\infty)=1$. Thus $\mathcal{D}\cap [0,n]$ is finite with probability 1 for each $n\ge 1$. Hence  $P(\underset{n>1}{\cap}\{\omega\ \mid  \mathcal{D}\cap [0,n] \textrm{ is finite} \})=1$. Therefore, $\mathcal{D}$ has no limit point in $\mathbb{R}$  w.p. $1.$\\
\end{proof}
%  \bibliography{A.bib}
%  \bibliographystyle{siam}

\end{document}